\documentclass[a4paper, 10pt]{extarticle}
\usepackage[top=70pt,bottom=70pt,left=75pt,right=77pt]{geometry}

\usepackage{amssymb, amsmath, amsthm, setspace, bbm, prettyref, graphicx, float, subfigure, color, mathrsfs, mathtools}
\usepackage{Lutsko_Style}
\DeclareMathSizes{10}{9}{7}{5}

\title{Long-Range Correlations of Sequences Modulo $1$}
\author{
{\sc Christopher Lutsko$^{\ast}$}
}

\begin{document}

  \maketitle
  \begin{abstract}
    \noindent In this paper we consider the fractional parts of a general sequence, for example the sequence $\alpha \sqrt{n}$ or $\alpha n^2$. We give a general method, which allows one to show that long-range correlations (correlations where the support of the test function grows as we consider more points) are Poissonian. We show that these statements about convergence can be reduced to bounds on associated Weyl sums. In particular we apply this methodology to the aforementioned examples. In so doing, we recover a recent result of Technau-Walker (2020) for the triple correlation of $\alpha n^2$ and generalize the result to higher moments. For both of the aforementioned sequences this is one of the only results which indicates the pseudo-random nature of the higher level ($m \ge 3$) correlations.


  \end{abstract}

  \renewcommand{\thefootnote}{*}
  \footnotetext{\noindent Rutgers University \newline Email: chris@lutsko.com}

  \onehalfspacing
  \setlength{\abovedisplayskip}{1mm}

  \section{Introduction}
\label{s:Introduction}

Let $\{a(n)\}_{n \in \N}$ be a sequence in $\R$. For a long time, mathematicians have studied the distribution of the fractional parts of such sequences. That is, the sequence $\{x_n\}_{n\in \N}$ given by

\begin{align*}
  x_n := a(n) \quad \Mod{1}.
\end{align*}
In general, it is fairly well-understood which sequences are uniformly distributed on the interval (see \cite{KuipersNiederreiter1974}). However, uniform distribution is a relatively crude measure of pseudo-randomness. Recently, there has been a lot of interest in the fine-scale local statistics of such sequences. Namely, a central question is to understand when the distribution of gaps between neighboring points in these sequences converges to the exponential function, as they do for uniformly distributed random variables on the interval.

This problem is fundamental from a mathematical point of view in understanding the random nature of deterministic sequences. Furthermore, it also has implications in physics. Since such sequences can be used to understand the energy levels of quantum systems, understanding their distribution has numerous implications. For example, the sequence $a(n) = \alpha n^2$ for $\alpha \neq 0$ corresponds to the eigenvalues of a boxed harmonic oscillator. In general, the well-known Berry-Tabor \cite{BerryTabor1977} states that for generic surfaces, if the Hamiltonian dynamics on the surface are integrable, then the gap distribution for the spectrum has a well-defined limit, given by the exponential distribution. There exist counter-examples to this statement which can be explained away (thus the word generic), however in general this relationship has been confirmed by experiment. But very little can be rigorously proved. For more information see the review \cite{Marklof2000}.

\subsection{Long-Range Correlations}

In general the gap distribution is very hard to work with directly. This owes to the fact that neighboring points are not necessarily consecutive points in the sequence. To get around this problem, we instead consider the $m$-level correlations: Given a sequence $\{x_n\}_{n \in \N} \subset [0,1)$, and a vector $\vect{j} \in \Z^m$, let $\Delta(\vect{j}) \in \R^{m-1}$ denote the difference vector

  \begin{align*}
      \Delta(\vect{j}) : = \left( x_{j_1} - x_{j_2}, x_{j_2} - x_{j_3}, \dots , x_{j_{m-1}} - x_{j_m}\right).
  \end{align*}
  Then for $f \in C_{c}^{\infty}(\R^{m-1})$ we define the \emph{$m$-level correlation} to be

 \begin{align} \label{corr}
   R^{(m)}(N,f) := \frac{1}{N} \sum_{\vect{j} \in \{1,\dots, N\}^m}^{\ast} f(N(\Delta(\vect{j}))), 
 \end{align}
 where the notation $\displaystyle \sum^\ast$ indicates that all entries of the vector $\vect{j}$ are distinct. We say that the $m$-level correlation is \emph{Poissonian} if 
\begin{align} \label{Poissonian}
  R^{(m)}(N,f) \to \int_{\R^{m-1} } f(\vect{x} ) d\vect{x}
\end{align}
in the limit as $N \to \infty$. And we say the sequence is \emph{Poissonian} if all the correlations for $m\ge 2$ are Poissonian. By the method of moments, it can be shown that if a sequence is Poissonian then the gap distribution is exponential. Therefore it is common to study the correlations rather than the gap distribution itself. 

That said, it remains a very difficult problem to show that a given sequence is Poissonian. While there are exceptions (e.g Rudnick-Zaharescu showed that almost every dilate of a lacunary sequence is Poissonian \cite{RudnickZaharescu2002}), there are very few results in the area. For example, if a sequence grows with a power law, then very little is known (especially when considering higher power correlations $(m \ge 3)$). Rather than study the correlations, in this paper we consider a coarser measure -- the \emph{long-range correlations:} fix $ \tau \in(0,1)$ 

\begin{align} \label{dil corr}
  R^{(m)}(N,f,\tau) := \frac{1}{N} \sum_{\vect{j} \in \{1,\dots, N\}^m}^{\ast} f(N^\tau(\Delta(\vect{j}))).
\end{align}
Hence, in comparison with the standard correlations, we are increasing the support of the function $f$ as $N \to \infty$. For uniformly distributed random variables on the interval one would expect:

\begin{align} \label{dil Poissonian}
  R^{(m)}(N,f,\tau) = N^{(m-1)(1-\tau)} \left(\int_{\R^{m-1}} f(\vect{x}) d\vect{x} + o(1)\right).
\end{align}
In this paper, we give a  general method to show that (for particular values of $\tau >0$) these long-range correlations converge to this limit, and thus coincide with the Poissonian limit. We show that these long-range correlations can be expressed in terms of the Poissonian limit \eqref{dil Poissonian} and an error which can be written in terms of Weyl sums. Then if we can appropriately bound these Weyl sums we can show that the error is small. 

It should be noted that if the discrepancy of a sequence goes to $0$ fast enough, then one can prove that the long-range correlations converge to the Poissonian limit for $\tau < \frac{1}{2}-\epsilon$ for all $\epsilon >0$ (see \cite{TechnauWalker2020} for details on this relation). The methodology in this paper, when applied to our examples, will improve on this 'na\"{i}ve' bound.

\subsection{Moments}

The long-range correlations are the most natural object to work with from a mathematical point of view. However from a intuitive point of view one can also consider the following random variable:

\begin{align} \label{W def}
  W_{L,N}=W_{L,N}(Y) : = \#\{n \le N : x_n \in [Y,Y+L/N] \mod{1}\}.
\end{align}
That is, we consider the number of points in a randomly placed interval of size $\frac{L}{N}$. Using a standard technique (see \cite{TechnauWalker2020}) one can prove the following proposition

\begin{proposition}\label{prop:moments}
  Fix a sequence $\{x_n \}_{n\in \N}$ and a moment $m \ge 2$. Then for a given $\tau < 1$, suppose that for all $f\in C_c^\infty(\R^{m-1})$ we have

  \begin{align*}
    R^{(m)}(N,f,\tau) = N^{(m-1)(1-\tau)}\left(\int_{\R^{m-1}}f(\vect{x})d\vect{x} +o(1)\right)
  \end{align*}
  as $N \to \infty$. Then

  \begin{align}
    \expect{W_{N,L}^m} = L^m (1+o(1)),
  \end{align}
  as $N \to \infty$, where $L = N^{1-\tau}$.
\end{proposition}

In words, this shows that, if we can show that the $m$-level long-range correlation converges to the Poissonian limit, then we can show that the $m^{th}$ moment of $W_{N,L}$ is (to leading order) $L^m$, as is the case for uniformly distributed random variables.

\subsection{Specific Examples}
In Section \ref{s:Methodology} we explain the general methodology for proving convergence as in \eqref{dil corr} however, prior to this we give two applications to well-studied sequences. 

\subsubsection{The Sequence $a(n) = \alpha n^2$}

If we consider $a(n) = \alpha n^2$, then the gaps between the sequence $x_n$ are related to the gaps in the energy levels of the ``boxed oscillator'', owing to this fact, and the sequence's importance mathematically, this case has been extensively studied \cite{RudnickSarnak1998, RudnickSarnakZaharescu2001, Heath-Brown2010}. In particular Rudnick and Sarnak showed that there is a set of $\alpha$ of full Lebesgue measure for which the pair correlation ($m=2$) is Poissonian. However very little is known about higher level correlations. 

We say $\alpha \in \R$ is Diophantine if, for every $\epsilon >0$ there is a $c=c(\alpha)>0$ such that
\begin{align}\label{Dio}
  \abs{\alpha - \frac{p}{q}} > \frac{c}{q^{2+\epsilon}}
\end{align}
for all integers $p,q$. Our first result, shows that for $\alpha$ Diophantine and for $\tau$ lying in a particular range, the long-range correlations for $a(n)=\alpha n^2$ are Poissonian:

\begin{theorem} \label{thm:an^2}
  Let $a(n) =\alpha n^2$ for $\alpha$ Diophantine, fix $m \ge 2$ and $\epsilon>0$, then if  $ 0<  \tau \le \frac{m}{2m-2} -\epsilon$ the long-range $m$-level correlation is Poissonian. Namely:
  \begin{align} \label{n^2 Poi}
    R^{(m)}(N,f,\tau) = N^{(m-1)(1-\tau)} \left(\int_{\R^{m-1}}f(\vect{x})d\vect{x} + o(1)\right),
  \end{align}
  as $N \to \infty$.
\end{theorem}

By Dirichlet's approximation theorem, the set of Diophantine $\alpha$ has full Lebesgue measure.

Theorem \ref{thm:an^2} has a number of implications. First of all, for $m=3$ this recovers a Theorem proved by Technau and Walker \cite{TechnauWalker2020}. Namely that the long-range correlations are Poissonian for $\tau < \frac{3}{4}-\epsilon$. In addition,  for $m \ge 4$ this is the first result proving convergence of long-range correlations that goes beyond the discrepancy bound of $\tau = \frac{1}{2}-\epsilon$.

Interestingly, our result holds for all Diophantine $\alpha$. In \cite{RudnickSarnak1998}, Rudnick and Sarnak give an example of a Diophantine $\alpha$ for which the (not long-range) correlations are not Poissonian. Theorem \ref{thm:an^2} thus shows that for such examples, the non-Poissonian behavior happens on a very fine scale which is $o(N^{-\tau})$.

\subsubsection{The Sequence $a(n) = \alpha \sqrt{n}$}

A second sequence of particular interest is the sequence $a(n) = \alpha \sqrt{n}$. If we consider $a(n) = \alpha n^\beta$ for $0<\beta < 1$, then it is expected that for $\beta \neq 1/2$, the gap statistics converge to the exponential distribution. In contrast, in a surprising paper, Elkies and McMullen \cite{ElkiesMcMullen2004} used homogeneous dynamics methods to show that if $\alpha^2 \in \Q$ and we consider $a(n)=\alpha \sqrt{n}$, then the gap distribution converges to an explicit distribution which is not the exponential. Moreover they conjectured that for $\alpha^2 \not\in \Q$ the gap distribution \emph{is} exponential. Also surprisingly El-Baz-Marklof-Vinogradov \cite{El-BazMarklofVinogradov2015a} showed that for $\alpha =1$ the pair correlation \emph{is} Poissonian. Therefore, in short, $\beta = 1/2$ is a surprising, special case exhibiting unusual behavior.

 \begin{theorem} \label{thm:a sqrt(n)}
   Let $a(n) = \alpha \sqrt{n}$ for any $\alpha \in \R^\ast$, fix any $\epsilon >0$. Take $\tau \le \frac{3m}{6m-4}-\epsilon$. Then

   \begin{align}\label{sqrt(n)}
     R^{(m)}(N,f,\tau) = N^{(m-1)(1-\tau)}\left(\int_\R f(x)dx +o(1)\right),
   \end{align}
   as $N\to \infty$.
   
 \end{theorem}

 Note that Theorem \ref{thm:a sqrt(n)} holds for any $\alpha \in \R^\ast$. Thus for $m=2$, \eqref{sqrt(n)} agrees with \cite{El-BazMarklofVinogradov2015a}, while for $m \ge 3$, \eqref{sqrt(n)} states that on this long-range scale all the correlations converge to the Poissonian limit, regardless of the value of $\alpha$. Therefore, this shows that the structure observed by Elkies-McMullen occurs on a very fine scale and isn't apparent on these long-range scales.

\vspace{2mm}

\noindent \textbf{Notation:} To avoid confusion, fix the following notation, let $f(n) = \cO(g(n))$ if   $\overline{\lim}_{n \to \infty}\abs{f(n)/g(n)} < \infty$, and $f(n) = o(g(n))$ if $\overline{\lim}_{n \to \infty} \abs{f(n)/g(n)} = 0$. Equivalently let $f(n) \ll g(n)$ denote $f(n) = \cO(g(n))$. 

For $x \in \R$ let $ \{x \}$ denote the fractional part of $x$, and we let $\|x\|$ denote the distance to the nearest integer. Given a set $\mathscr{A}$ we denote $\mathscr{A}^\ast := \mathscr{A} \setminus \{0\}$. Finally, as usual $e(z) := e^{2\pi i z}$.

  \section{Methodology}
\label{s:Methodology}

To achieve the convergence in \eqref{dil Poissonian} we first use discrete Fourier analysis, to show that the long-range $m$-level correlation can be written as a main term and a remainder which can be written as a product of Weyl sums. Then, for each example we use bounds on the associated Weyl sums to show that the remainder is sufficiently small. For $\vect{y} \in \R^{m}$, write

\begin{align}
  \vect{\beta}(\vect{y}) = (\beta_1(\vect{y}),\dots, \beta_{m-1}(\vect{y}))
\end{align}
where $\beta_i(\vect{y}) := a(y_{i+1}) -a(y_{i})$. Thus 

\begin{align}
  R^{(m)} (N,f,\tau) = \frac{1}{N}\sum_{\vect{y}\in \{1,\dots,N\}^m}^\ast \sum_{\vect{k} \in \Z^{m-1}} f(N^\tau(\vect{\beta}(\vect{y}) +\vect{k})).
\end{align}
Now apply Poisson summation to the sum over $\vect{k}$ to give:

\begin{align}
  R^{(m)} (N,f,\tau) = \frac{1}{N}\sum_{\vect{y}\in \{1,\dots,N\}^m}^\ast \sum_{\vect{k} \in \Z^{m-1}} e(\vect{k} \cdot \vect{\beta}(\vect{y})) \int_{\R^{m-1}}f(N^\tau \vect{x})e(\vect{x}\cdot \vect{k})d\vect{x}.
\end{align}
Now we isolate the term $\vect{k} = 0$ which gives us our main term:

\begin{align}
  R^{(m)} (N,f,\tau) = N^{(1-\tau)(m-1)}\left(\int f(\vect{x})d\vect{x} + o(1)\right) +  \cE
\end{align}
where the $o(1)$-error comes from the fact that the $y_i$ are taken to be distinct and

\begin{align*}
  \cE :=  \frac{1}{N}\sum_{\vect{y}\in \{1,\dots,N\}^m}^\ast \sum_{0 \neq \vect{k} \in \Z^{m-1}} e(\vect{k} \cdot \vect{\beta}(\vect{y})) \int_{\R^{m-1}}f(N^\tau \vect{x})e(\vect{x}\cdot \vect{k})d\vect{x}.
\end{align*}
Thus our goal in the remainder of the paper is to show that $\cE = o \left(N^{(1-\tau)(m-1)}\right)$. To achieve this bound let $M = N^{\tau+\epsilon^\prime}$ for some $\epsilon^\prime>0$ which we fix later, then we can use the fast decay of Fourier coefficients to show:

\begin{align*}
  \cE &=  \frac{1}{N}\sum_{\vect{y}\in \{1,\dots,N\}^m}^\ast \sum_{\substack{0 \neq \vect{k} \in \Z^{m-1}\\ \abs{k_i} < M }} e(\vect{k} \cdot \vect{\beta}(\vect{y})) \int_{\R^{m-1}}f(N^\tau \vect{x})e(\vect{x}\cdot \vect{k})d\vect{x}  + o(1)\\
  & \ll \frac{1}{N^{1+\tau(m-1)}} \sum_{\substack{0 \neq \vect{k} \in \Z^{m-1}\\ \abs{k_i} < M }} \abs{ \sum_{\vect{y}\in \{1,\dots,N\}^m}^\ast e(\vect{k} \cdot \vect{\beta}(\vect{y}))} + o(1).
\end{align*}
For notation set $k_0=k_m = 0$, then we can change variables and write:

\begin{align*}
  \cE &\ll  \frac{1}{N^{1+\tau(m-1)}} \sum_{\substack{0 \neq \vect{k} \in \Z^{m-1}\\ \abs{k_i} < M }} \abs{ \sum_{\vect{y}\in \{1,\dots,N\}^m}^\ast \prod_{i=1}^{m} e(a(y_i) (k_{i-1}-k_i))} + o \left(N^{(m-1)(1-\tau)}\right)\\
      & =  \frac{1}{N^{1+\tau(m-1)}} \sum_{\substack{0 \neq \vect{k} \in \Z^{m-1}\\ \abs{k_i} < 2M }} \abs{ \sum_{\vect{y}\in \{1,\dots,N\}^m}^\ast e(a(y_{m})d(\vect{k}))\prod_{i=1}^{m-1} e(a(y_i)k_i)} + o \left(N^{(m-1)(1-\tau)}\right)
\end{align*}
where $d(\vect{k})= -k_1-k_2-\dots - k_{m-1}$.

If we were considering the full correlations $(\tau =1)$, then the diagonal terms $y_i = y_j$ ($i \neq j$) would be of the same size as the main term. However since we work with $\tau < 1$, the diagonal terms are in fact negligible. Thus we can add them back into the sum. To that end, we use the inclusion-exclusion principle to write

\begin{align*}
   \sum_{\vect{y}\in \{1,\dots,N\}^m}^\ast = \sum_{\vect{y}\in \{1,\dots,N\}^m} - \sum_{\vect{y}\in \{1,\dots,N\}^m}^{(2)} 
\end{align*}
where the $(2)$ above sum signifies that at least $2$ terms in the sum must be equal. By symmetry:

\begin{align}
  \begin{aligned}
  &\frac{1}{N^{1+\tau(m-1)}} \sum_{\substack{0 \neq \vect{k} \in \Z^{m-1}\\ \abs{k_i} < 2M }} \abs{ \sum_{\vect{y}\in \{1,\dots,N\}^m}^{(2)} e(a(y_{m})d(\vect{k}))\prod_{i=1}^{m-1} e(a(y_i)k_i)} \\
  &\phantom{++++}\ll \frac{M^2}{N^{1+\tau(m-1)}} \sum_{\substack{\vect{k} \in\Z^{m-3}\\ \abs{k_i} < 2M }} \abs{ \sum_{\vect{y}\in \{1,\dots,N\}^{m-2}} e(a(y_{m-2})d(\vect{k}))\prod_{i=1}^{m-1-2} e(a(y_i)k_i)}\\
    &\phantom{++++}\ll  \frac{M^2}{N^{1+\tau(m-1)}} \sum_{\substack{0\neq \vect{k} \in\Z^{m-3}\\ \abs{k_i} < 2M }} \abs{ \sum_{\vect{y}\in \{1,\dots,N\}^{m-j}} e(a(y_{m-2})d(\vect{k}))\prod_{i=1}^{m-3} e(a(y_i)k_i)} + \frac{M^2}{N^{1+\tau(m-1)}} N^{m-2}
  \end{aligned}
\end{align}
If we write:
  \begin{align*}
    \cE_m : = \frac{1}{N^{1+\tau(m-1)}} \sum_{\substack{0 \neq \vect{k} \in \Z^{m-1}\\ \abs{k_i} < 2M }} \abs{ \sum_{\vect{y}\in \{1,\dots,N\}^m} e(a(y_{m})d(\vect{k}))\prod_{i=1}^{m-1} e(a(y_i)k_i)}
  \end{align*}
  Then we have shown that, for $\tau <1$

\begin{align}\label{inc exc}
  \cE \ll  \cE_m + \frac{M}{N^\tau}\cE_{m-1}  + o \left(N^{(m-1)(1-\tau)}\right).
\end{align}

Now we write

\begin{align*}
  S(N,k) : = \sum_{y=1}^N e(ka(y)). 
\end{align*}
In this case

\begin{align}\label{E bound}
  \cE_m &\ll   \frac{1}{N^{1+\tau(m-1)}} \sum_{\substack{0 \neq \vect{k} \in \Z^{m-1}\\ \abs{k_i} < 2M }} \abs{S(N,d(\vect{k}))}\prod_{i=1}^{m-1}\abs{S(N,k_i)} 
\end{align}
Thus we get the following theorem:

\begin{theorem} \label{thm:Weyl Bound}
  Let $\tau < 1$, then for $m \ge 2$:
  \begin{align}\label{Weyl m}
    \frac{\cE_m}{N^{(m-1)(1-\tau)}} \ll \frac{1}{N^m} \sum_{\substack{0 \neq \vect{k} \in \Z^{m-1}\\ \abs{k_i} < 2M }} \abs{S(N,d(\vect{k}))}\prod_{i=1}^{m-1}\abs{S(N,k_i)}.
  \end{align}

\end{theorem}  
In what remains, we use Theorem \ref{thm:Weyl Bound} to show that (under the hypotheses stated in the main theorems of the introduction) $\cE_m = o(1)$. Thus, since the range of $\tau$ decreases with $m$, we can use an inductive argument to prove our main theorems.

\begin{remark}
  Unfortunately Theorem \ref{thm:Weyl Bound}, as it is stated, is not powerful enough to handle the case $\tau =1$ (i.e the standard correlations). The reason for this is that in that case the diagonal terms are of leading order. Therefore, to use the same approach would require a precise asymptotic for the Weyl sums $S(N,k)$. These kinds of precise asymptotics are rare in the field. Most likely, to handle the full correlations would require a multi-dimensional analysis as is used in \cite{TechnauWalker2020}.
\end{remark}


\section{Proof of Theorem \ref{thm:an^2}}
\label{s:proof of an2}

The proof of Theorem \ref{thm:an^2} is more-or-less a straightforward application of Theorem \ref{thm:Weyl Bound} and Weyl differencing (see e.g {\cite[Lemma 3.1]{Davenport2005}}). Namely, as was shown in {\cite[Corollary 5]{RudnickSarnak1998}} for $\alpha$ Diophantine, we have, for any $\delta >0$:

\begin{align}
  \label{S bound 1}
  \sum_{1\le k \le M} \abs{S(k,N)}^2 \ll M^{1+\delta}N^{1+\delta}\\
  \label{S bound 2}
  \sum_{1\le k \le M} \abs{S(k,N)} \ll M^{1+\delta}N^{1/2+\delta}.
\end{align}
Therefore if we consider the r.h.s in Theorem \ref{thm:Weyl Bound}: since $\vect{k} \neq 0$, we can assume $k_1 \neq 0$:

\begin{align*}
  \frac{1}{N^m}\sum_{\substack{0 \neq \vect{k}\in \Z^{m-1}\\\abs{k_i}< 2M}}\abs{S(N,d(\vect{k}))} \prod_{i=1}^{m-1}\abs{S(N,k_i)} \ll  \frac{1}{N^m}\sum_{\substack{\vect{k}^\prime \in \Z^{m-2}\\\abs{k_i}< 2M}} \prod_{i=2}^{m-1}\abs{S(N,k_i)} \sum_{k_1=1}^M \abs{S(N,k_1)}\abs{S(N,d(\vect{k}))}
\end{align*}
where $\vect{k}^\prime = (y_2,\dots, y_{m-1})$. First we handle the term $d(\vect{k})=0$. Note that setting $d(\vect{k}) = 0$, fixes the value of $k_1 = d(\vect{k}^\prime) = -k_2-k_3-\dots -k_{m-1}$. Moreover, we know that $k_1 >0$, thus (applying the trivial bound $S(N,0)=N$),

\begin{align*}
  \frac{1}{N^m}\sum_{\substack{\vect{k}^\prime \in \Z^{m-2}\\\abs{k_i}< 2M}} \prod_{i=2}^{m-1}\abs{S(N,k_i)} \sum_{k_1=1}^M \abs{S(N,k_1)}\abs{S(N,0 )}
  &= 
  \frac{1}{N^{m}}\sum_{\substack{0 \neq \vect{k}^\prime \in \Z^{m-2}\\\abs{k_i}< 2M}}  \abs{S(N,d(\vect{k}^\prime)}\abs{S(N,0)}\prod_{i=2}^{m-1}\abs{S(N,k_i)}\\
  &=  \frac{1}{N^{m-1}}\sum_{\substack{0 \neq \vect{k}^\prime \in \Z^{m-2}\\\abs{k_i}< 2M}} \abs{S(N,d(\vect{k}^\prime))} \prod_{i=2}^{m-1}\abs{S(N,k_i)}\\
  & = \frac{1}{N^{m-1}}\sum_{\substack{0 \neq \vect{k}^\prime \in \Z^{m-2}\\\abs{k_i}< 2M}} \abs{S(N,d(\vect{k}^\prime))} \prod_{i=2}^{m-1}\abs{S(N,0)}
\end{align*}
Thus we can use an inductive argument to bound the $d(\vect{k})=0$ term (i.e this is the same as the right hand side in Theorem \ref{thm:Weyl Bound} for $\cE_{m-1}$). For the other term, we apply Cauchy Schwarz to the sum over $k_1$ giving:

\begin{align*}
  \frac{1}{N^m}\sum_{\substack{0 \neq \vect{k}\in \Z^{m-1}\\\abs{k_i}< 2M\\ d(\vect{k})\neq 0}}\abs{S(N,d(\vect{k}))} \prod_{i=1}^{m-1}\abs{S(N,k_i)} \ll  \frac{1}{N^m}\sum_{\substack{\vect{k}^\prime \in \Z^{m-2}\\\abs{k_i}< 2M}} \prod_{i=2}^{m-1}\abs{S(N,k_i)} \left(\sum_{k=1}^M\abs{S(N,k)}^2 \sum_{k=1}^{mM}\abs{S(N,k)}^2\right)^{1/2},
\end{align*}
where we have used the fact that $0< \abs{d(\vect{k})} \le mM$ for any value of $k_1$. Now we can use \eqref{S bound 1} on the final bracket, and for each of the remaining $k_i$-sums we use \eqref{S bound 2} and the trivial bound $S(N,0) = N$. Thus

\begin{align}
  \begin{aligned}\label{error n2}
  \frac{1}{N^m}\sum_{\substack{0 \neq \vect{k}\in \Z^{m-1}\\\abs{k_i}< 2M}}\abs{S(N,d(\vect{k}))} \prod_{i=1}^{m-1}\abs{S(N,k_i)} &\ll  \frac{1}{N^m}\left(N + M^{1+\delta}N^{1/2+\delta}\right)^{m-2} \left(M^{1+\delta}N^{1+\delta}\right)\\
  &\ll \frac{1}{N^m} M^{(1+\delta)(m-1)}N^{m(1/2+\delta)} + \frac{M^{1+\delta}N^\delta}{N}\\
  &\ll \frac{M^{m-1+\delta^\prime}}{N^{m/2}}  + \frac{M^{1+\delta}N^\delta}{N}.
  \end{aligned}
\end{align}
Recall, we want to show $\frac{\cE_m}{N^{(m-1)(1-\tau)}} = o \left(1\right)$. Therefore, provided $M = N^{\frac{m}{2m-2}-\epsilon}$, for some $\epsilon > \delta^\prime \left(\frac{m}{2(m-1)(m-1+\delta^\prime)}\right)$, then the l.h.s of \eqref{error n2} is $o(1)$. Inserting this bound into \eqref{Weyl m} and using \eqref{inc exc} proves Theorem \ref{thm:an^2}.

\qed

  \section{Proof of Theorem \ref{thm:a sqrt(n)}}
\label{s:sqrt(n)}

For the proof of Theorem \ref{thm:a sqrt(n)}, we have not found any suitable bounds on Weyl sums in the literature. Fortunately we can use the established theory of exponential sums to prove the following:

\begin{theorem} \label{thm:exponential sums}
  Let $0< \abs{k} \le CM$ for any constant $C<\infty$, then as $N \to \infty$:

  \begin{align} \label{exponential}
    S(N,k) \ll k^{1/2}N^{1/4} + \frac{N^{3/4}}{k^{1/2}}\log(N).
  \end{align}

\end{theorem}
With that, Theorem \ref{thm:a sqrt(n)} follows almost immediately from \eqref{Weyl m}. First, note that: for any fixed constant $C>0$

\begin{align}
  &\sum_{0 < \abs{k_i} < CM } \abs{S(N,k)}^2 \ll M^2N^{1/2} + N^{3/2}\log(N)^2 + N M\log(N)\\
  &\sum_{0 < \abs{k_i} < CM } \abs{S(N,k)} \ll M^{3/2}N^{1/4} + N^{3/4}M^{1/2}\log(N)
\end{align}
as $N \to \infty$.

Just as we did in Section \ref{s:proof of an2}, we note that if $d(\vect{k})=0$ then $k_1 = -k_2-\dots -k_{m-1} = d(\vect{k}^\prime)$. Thus

\begin{align*}
  \frac{1}{N^m} \sum_{\substack{0 \neq \vect{k} \in \Z^{m-1}\\ \abs{k_i} < 2M }} \abs{S(N,d(\vect{k}))}\prod_{i=1}^{m-1}\abs{S(N,k_i)}  &\ll \frac{1}{N^m} \sum_{\substack{\vect{k}^\prime \in \Z^{m-2}\\  \abs{k_i} < 2M }}\prod_{i=2}^{m-1}\abs{S(N,k_i)} \sum_{k_1=1}^{2M}\abs{S(N,k_1)}\abs{S(N,d(\vect{k}))},\\
  &\ll \frac{1}{N^m} \sum_{\substack{\vect{k}^\prime \in \Z^{m-2}\\  \abs{k_i} < 2M }}\prod_{i=2}^{m-1}\abs{S(N,k_i)} \sum_{\substack{k_1=1\\d(\vect{k})\neq 0}}^{2M}\abs{S(N,k_1)}\abs{S(N,d(\vect{k}))}\\
  & \phantom{++++++} +  \frac{1}{N^{m-1}} \sum_{\substack{0 \neq\vect{k}^\prime \in \Z^{m-2}\\  \abs{k_i} < 2M }}\prod_{i=2}^{m-1}\abs{S(N,k_i)} \abs{S(N,d(\vect{k}^\prime)}.
\end{align*}
Thus (as before), we can use an inductive argument on $m$ to show that the latter term in the last line is sufficiently small. For the former term we apply Cauchy-Schwarz to get

\begin{align*}
  \frac{1}{N^m} \sum_{\substack{\vect{k}^\prime \in \Z^{m-2}\\  \abs{k_i} < 2M }}\prod_{i=2}^{m-1}\abs{S(N,k_i)} \sum_{\substack{k_1=1\\d(\vect{k})\neq 0}}^{2M}\abs{S(N,k_1)}\abs{S(N,d(\vect{k}))}&\ll \frac{1}{N^m} \sum_{\substack{\vect{k}^\prime \in \Z^{m-2}\\  \abs{k_i} < 2M }}\prod_{i=2}^{m-1}\abs{S(N,k_i)} \sum_{k_1=1}^{mM}\abs{S(N,k_1)}^2\\
  & \ll \frac{1}{N^m} \left(M^{3/2}N^{1/4}\right)^{m-2}\left(N^{1/2}M^{2}\right)\\
  & \ll \frac{1}{N^{\frac{3m}{4}}}M^{\frac{3m-2}{2}} 
\end{align*}
(we have assumed here that $\tau \ge \frac{1}{2}$). Now inserting $M= N^\tau$ and recalling that we want $\frac{\cE_m}{N^{(m-1)(1-\tau)}} = o \left(1\right)$, this implies:

\begin{gather*}
  \frac{3m-2}{2}\tau   - \frac{3m}{4} < 0 \qquad \Longrightarrow \qquad
  \tau   <  \frac{3m}{2(3m-2)}
\end{gather*}
as in Theorem \ref{thm:a sqrt(n)}.

\qed

  \section{Exponential Sum Bounds}
\label{s:Exponential Sum Bounds}

This section is devoted to the proof of Theorem \ref{thm:exponential sums}, which is based on Van der Corput's B-Process -- see {\cite[Chapter 3]{Montgomery1994}} for an excellent reference. However, for our problem, the range of the sum $[1,N]$ is very large. As a consequence we do not have good control on the derivatives of the exponent. Hence we need to prove the stationary phase integrals rather than cite a known reference. Then we will use some classical analytic tools to control the various terms which arise. For simplicity of notation we assume $k >0$. Then to achieve the same bounds for negative $k$ take complex conjugates.

\subsection{Analytic Tools}
\label{ss:Analytic Tools}

For a real-valued function $g(x)$ defined on $[A,B]$ we define:

\begin{align}\label{V def}
  V(g) := \max_{(A,B)} \abs{g(x)} + t.v(g(x)),
\end{align}
where $t.v(g(x)) := \int_A^B \abs{g^\prime(x)}dx$. For a proof of the following classical lemma see {\cite[Lemma 5.1.4]{Huxley1996}}

\begin{lemma}[Van der Corput's $k^{th}$ Derivative Test]\label{lem:k deriv}
  Let $f(x)$ be real and $k$ times differentiable on $(A,B)$ with $f^{(r)}(x) \ge \mu>0$ on $(A,B)$. Let $g(x)$ be a real valued function. Then:

  \begin{align}\label{k deriv}
    \abs{\int_A^B g(x) e(f(x))dx } \ll \frac{V(g)}{\mu^{1/r}}.
  \end{align}
\end{lemma}

\subsubsection{Stationary Phase Integrals}
\label{sss:Stationary Phase Integrals}

If $\frac{k\alpha}{2\sqrt{N}} \le  r \le \frac{k\alpha}{2}$, then the function $h_{k,r}(x) = k\alpha\sqrt{x} - rx$ has a stationary point, $\gamma_{k,r} = \left(\frac{k\alpha}{2r}\right)^2\in (1,N)$ where  $h_{k,r}^\prime(\gamma_{k,r}) = 0 $. In that case the following is an application of classical stationary phase estimates (see for example {\cite[Lemma 5.5.2]{Huxley1996}}):

\begin{proposition}[Stationary Phase Integral]\label{prop:SPI}
  Let $r,k\in \Z$ be integers such that, $k \in \{1,2,\dots, N\}$ and $\frac{k\alpha}{2\sqrt{N}}\le  r < \frac{k\alpha}{2} $. Then in the limit as $N \to \infty$:

    \begin{align}\label{stationary}
      \begin{aligned}
      \int_1^Ne\left(h_{k,r}(x)\right)dx 
         &= e\left(\frac{(k\alpha)^2}{4r} + \frac{1}{8}\right)\frac{\gamma_{k,r}^{3/4}}{k^{1/2}}  
         +\cO\left( \frac{1}{h_{k,r}^\prime(1)}    
         +\frac{1}{h_{k,r}^\prime(N)}\right)
         + \cO\left(\frac{N^{1/4}}{k^{3/2}}  + \frac{N^{1/2}}{k^2}   \right).
      \end{aligned}
    \end{align}

\end{proposition}

\begin{proof}
  For this proof we will suppress the subscripts $k,r$ whenever they appear. To prove \eqref{stationary} we fix two constants $0 < c_1 < 1 < c_2 < \infty$ independent of $N$ and fix $u :=  c_1\gamma$ and $v : = c_2\gamma$. Now we consider the three integrals on $(1,u)$, $(u,v)$, and $(v,N)$ separately.

\begin{steps}
\item To address the integral on $(1,u)$, apply integration by parts, to give:

  \begin{align}
    \begin{aligned}
    \int_{1}^u e(h(x))dx &= \left[\frac{e(h(x))}{2\pi i h^\prime(x)}\right]_{x=1}^u + \frac{1}{2\pi i} \int_{1}^u \frac{h^{\prime\prime}(x)e(h(x))}{h^\prime(x)^2}dx
    \end{aligned}
  \end{align}
  Now if $u < 1$, then the term of the right is $\cO \left(\frac{1}{h^\prime(1)}\right)$. If not, then $r < \frac{k\alpha\sqrt{c_1}}{2}$. Thus, applying the first integral test gives:

  \begin{align}
    \begin{aligned}
    \int_{1}^u e(h(x))dx & = \left[\frac{e(h(x))}{2\pi i h^\prime(x)}\right]_{x=1}^u + \cO \left( V\left( \frac{h^{\prime\prime}(x)}{h^\prime(x)^2} \right) \left(\min \abs{h^\prime(x)}\right)^{-1} \right)
    \end{aligned}
  \end{align}

  Now note that on $(1,u)$, the derivative $\abs{h^\prime(x)} \ge C\frac{k}{\gamma^{1/2}}$ where the constant $C= \alpha\left(\frac{1}{\sqrt{c_1}}-1\right)$. Thus

  \begin{align}\label{(1,u)}
    \begin{aligned}
    \int_{1}^u e(h(x))dx & = \left[\frac{e(h(x))}{2\pi i h^\prime(x)}\right]_{x=1}^u + \cO \left(\frac{\gamma^{1/2}}{k} \left[ \frac{h^{\prime\prime}(x)}{h^\prime(x)^2} \right]_1^u  \right)\\
    &= \left[\frac{e(h(x))}{2\pi i h^\prime(x)}\right]_{x=1}^u + \cO \left(\frac{\gamma^{1/2}}{k} \left( \frac{k}{\left(k\alpha - r\right)^2} + \frac{1}{\gamma^{1/2}k}   \right)\right)\\
    &= \left[\frac{e(h(x))}{2\pi i h^\prime(x)}\right]_{x=1}^u + \cO \left(\frac{N^{1/2}}{k^2} + \frac{1}{k^2}   \right).
    \end{aligned}
  \end{align}


    \item  The proof for $(v,N)$ is similar: First we use the first derivative test (Lemma \ref{lem:k deriv}) to show:
    \begin{align*}
      \int_v^Ne(h(x))dx&=
      \left[\frac{e(h(x))}{2\pi i h^\prime(x)}\right]_{x=v}^N + \frac{1}{2\pi i} \int_{v}^q \frac{h^{\prime\prime}(x)e(h(x))}{h^\prime(x)^2}dx\\
      & = \left[\frac{e(h(x))}{2\pi i h^\prime(x)}\right]_{x=v}^N + \cO\left(\frac{V\left(h^{\prime\prime}(x)/h^\prime(x)^2\right)}{\min \abs{h^\prime(x)}} \right)   \\
    \end{align*}
     Moreover, evaluating $V$ gives:

    \begin{align}
      \begin{aligned}
      \frac{V\left(h^{\prime\prime}(x)/h^\prime(x)^2\right)}{\min\abs{h^\prime(x)}} &= \cO\left(\frac{1}{k^2}\right).
      \end{aligned}
    \end{align}
    Therefore:

    \begin{align}\label{(v,q)}
      \int_v^N e(h(x))dx =  \left[\frac{e(h(x))}{2\pi i h^\prime(x)}\right]_{x=v}^N + \cO\left(\frac{1}{k^2}\right). 
    \end{align}

  \item
    Working in $[u,v]$, set $T = k\alpha \gamma^{1/2}$ and $M= \gamma$. Then for $j=2,3,4$, there exist constants $0< C_j < \infty$ such that

  \begin{align*}
    \abs{h^{(j)}(x)} \le C_j\frac{T}{M^j} \qquad , \qquad h^{\prime\prime}(x)\ge h^{\prime\prime}(v) \ge C_2^{-1}  \frac{T}{M^2}
  \end{align*}
  Thus we can apply a classical stationary phase integral {\cite[Lemma 5.5.2]{Huxley1996}} to say:

  \begin{align}
    \begin{aligned}
    \int_u^vh(x) dx &= \frac{e(h(\gamma)+1/8)}{\sqrt{h^{\prime\prime}(\gamma)}} + \frac{e(h(v))}{2\pi i h^{\prime}(v)} - \frac{e(h(u))}{2\pi i h^{\prime}(u)}
                      + \cO\left(\frac{M}{T^{3/2}} + \frac{M^4}{T^2}\left(\frac{1}{\gamma}  \right)^3  \right)\\
     &= \frac{e(h(\gamma)+1/8)}{\sqrt{h^{\prime\prime}(\gamma)}} + \frac{e(h(v))}{2\pi i h^{\prime}(v)} - \frac{e(h(u))}{2\pi i h^{\prime}(u)}
       + \cO\left(\frac{\gamma}{k^{3/2}\gamma^{3/4}} + \frac{\gamma^3}{k^2}\frac{1}{\gamma^3}  \right)\\
     &= \frac{e(h(\gamma)+1/8)}{\sqrt{h^{\prime\prime}(\gamma)}} + \frac{e(h(v))}{2\pi i h^{\prime}(v)} - \frac{e(h(u))}{2\pi i h^{\prime}(u)}
     + \cO\left(\frac{\gamma^{1/4}}{k^{3/2}} + \frac{1}{k^2} \right).
     \end{aligned}
  \end{align}
  \eqref{stationary}, now follows directly from this bound together with \eqref{(1,u)} and \eqref{(v,q)}. 

\vspace{-5mm}
\end{steps}
\end{proof}

\subsection{Proof of Theorem \ref{thm:exponential sums}}

Our goal is to estimate
\begin{align*}
  S = S(N,k) = \sum_{y\in \{1,\dots, N \} }e \left( k\alpha\sqrt{y}\right).
\end{align*}

\begin{steps}
  \item
  First, apply the truncated Poisson summation formula {\cite[Lemma 5.4.3]{Huxley1996}}: For $A := \frac{k\alpha}{2\sqrt{N}}$ and $B:=\frac{k\alpha}{2}$:
  
  \begin{align*}
    S = \sum_{A-\frac{1}{4} < r < B+\frac{1}{4}} \int_1^N e\left(k\alpha \sqrt{x} - rx\right)dx + \cO( \log (N)).
  \end{align*}
  First note that if $r=0$ then $\int_1^N e\left(k\alpha\sqrt{x}\right)dx \ll \frac{N^{1/2}}{k}$ by the first derivative test. Thus, if we set $\wt{A} = \min\left(1, \lceil A-1/4\rceil\right)$ and $\wt{B} :=\lfloor B+1/4\rfloor$, and apply Proposition \ref{prop:SPI} and Lemma \ref{lem:k deriv} ($2^{nd}$ derivative test) then we have the bound:

  \begin{align}
    \begin{aligned}\label{E expanded}
    S &  = \sum_{\wt{A} \le  r \le  \wt{B}} \left( e\left(\frac{(k\alpha)^2}{4r} + \frac{1}{8}\right)\frac{\gamma_{k,r}^{3/4}}{k^{1/2}}  
    +\cO\left( \min \left\{\frac{1}{ h_{k,r}^\prime(1)}, \frac{N^{3/4}}{k^{1/2}}\right\}
    + \min \left\{\frac{1}{h_{k,r}^\prime(N)}, \frac{N^{3/4}}{k^{1/2}}\right\}\right)\right)\\
    &\phantom{+++++++++++++++++++++++} + \cO\left( \frac{N^{1/2}}{k} + \frac{N^{1/4}}{k^{1/2}} + \log(N)\right)  .
    \end{aligned}
  \end{align}
  We have used the fact that there are order $k$ terms in the sum over $r$ to pull two of the error terms out of the sum (and the resulting error is $\cO(k^{1/2})$). 


\item
  The latter two terms in the sum can both be controlled in the same way:

\begin{align}
\begin{aligned}\label{step2}
\sum_{\wt{A} \le  r \le \wt{B} }\min \left\{\frac{1}{ h_{k,r}^\prime(1)}, \frac{N^{3/4}}{k^{1/2}}\right\} & \ll  \abs{\sum_{\wt{A} \le  r \le \wt{B}} \min \left\{\frac{1}{ \frac{k\alpha}{2} - r }, \frac{N^{3/4}}{k^{1/2}}\right\}    }      \\
              & \ll \log (N) + \min \left\{ \frac{1}{\| \frac{k\alpha}{2} \|}, \frac{N^{3/4}}{k^{1/2}}  \right\}.
\end{aligned}
\end{align}
The same bound holds for the term involving $h^\prime_{k,r}(N)$. Thus:

\begin{align} \label{S bound1}
  S &  = \sum_{\wt{A} \le  r \le \wt{B}}  e\left(\frac{(k\alpha)^2}{4r} + \frac{1}{8}\right)\frac{\gamma_{k,r}^{3/4}}{k^{1/2}}  
        + \cO\left( \frac{N^{3/4}}{k^{1/2}} + \frac{N^{1/2}}{k} +\log(N) \right).
\end{align}


\item

  Inserting the definition of $\gamma_{k,r}$, the remaining term can be written:

  \begin{align}
    \begin{aligned}
      E &:= k \sum_{\wt{A} < r <  \wt{B} }  e\left(\frac{(k\alpha)^2}{4r}\right)\frac{1}{r^{3/2}},\\
      &\le k \sum_{\wt{A}< r< \wt{B}} \frac{1}{r^{3/2}}\\
      &\ll k \frac{1}{\wt{A}^{1/2}} \ll N^{1/4}k^{1/2}.
      \end{aligned}
    \end{align}

\end{steps}
\qed

  \small 
  \section*{Acknowledgements}
  For part of this work, the author was supported by EPSRC Studentship EP/N509619/1 1793795. Furthermore, the author would like to thank Jens Marklof for several useful discussions and Niclas Technau for comments on an early preprint of the paper.

  \bibliographystyle{alpha}
  \bibliography{biblio}

\end{document}